\documentclass{article}
\usepackage[top=2cm, bottom=2cm, left=2cm, right=2cm]{geometry}

\usepackage[alphabetic,lite]{amsrefs}                                          
\usepackage{amssymb}                                                           
\usepackage{amsthm}                                                            
\usepackage{appendix}                                                          
\usepackage{array}                                                             
\usepackage{bm}                                                                
\usepackage{colonequals}                                                       
\usepackage{color}                                                             
\usepackage{enumitem}                                                          
\usepackage{graphicx}                                                          
\usepackage{indentfirst}                                                       
\usepackage{mathrsfs}                                                          
\usepackage{mathtools}                                                         
\usepackage{moreenum}                                                          
\usepackage{stmaryrd}                                                          
\usepackage{verbatim}                                                          
\usepackage[all,cmtip]{xy}                                                     

\definecolor{tianred}{rgb}{0.57, 0.36, 0.51}                                   
\definecolor{tianblue}{rgb}{0.0, 0.22, 0.66}                                   
\definecolor{tianpink}{rgb}{0.88, 0.56, 0.59}                                  
\definecolor{tiangreen}{rgb}{0.24, 0.82, 0.44}                                 

\usepackage[colorlinks,
            linkcolor=tianred,
            anchorcolor=tiangreen,
            citecolor=tianblue,
            urlcolor=tianpink]{hyperref}                                       
\usepackage{cleveref}                                                          

\usepackage[OT2,T1]{fontenc}
\DeclareSymbolFont{cyrletters}{OT2}{wncyr}{m}{n}
\DeclareMathSymbol{\RBe}{\mathalpha}{cyrletters}{"42}                          
\DeclareMathSymbol{\Che}{\mathalpha}{cyrletters}{"51}                          
\DeclareMathSymbol{\Sha}{\mathalpha}{cyrletters}{"58}                          

\makeatletter
\DeclareRobustCommand\widecheck[1]{{\mathpalette\@widecheck{#1}}}
\def\@widecheck#1#2{%
    \setbox\z@\hbox{\m@th$#1#2$}%
    \setbox\tw@\hbox{\m@th$#1%
       \widehat{%
          \vrule\@width\z@\@height\ht\z@
          \vrule\@height\z@\@width\wd\z@}$}%
    \dp\tw@-\ht\z@
    \@tempdima\ht\z@ \advance\@tempdima2\ht\tw@ \divide\@tempdima\thr@@
    \setbox\tw@\hbox{%
       \raise\@tempdima\hbox{\scalebox{1}[-1]{\lower\@tempdima\box
\tw@}}}%
    {\ooalign{\box\tw@ \cr \box\z@}}}
\makeatother

\theoremstyle{plain}      \newtheorem{thm}{Theorem}[section]                   
\theoremstyle{plain}                   
\theoremstyle{plain}      \newtheorem{lem}[thm]{Lemma}                         
\theoremstyle{plain}                             
\theoremstyle{plain}      \newtheorem{cor}[thm]{Corollary}                     
\theoremstyle{plain}                         
\theoremstyle{plain}                        
\theoremstyle{plain}      \newtheorem{conjecture}[thm]{Conjecture}             
\theoremstyle{definition} \newtheorem{rmk}[thm]{Remark}                        
\theoremstyle{definition}                      
\theoremstyle{definition} \newtheorem{df}[thm]{Definition}                     
\theoremstyle{definition}                   
\theoremstyle{definition} \newtheorem{eg}[thm]{Example}                        
\theoremstyle{definition}                        
\theoremstyle{definition}                        
\theoremstyle{definition}                      
\theoremstyle{definition}                    
\theoremstyle{definition}                  
\theoremstyle{definition}                        
\theoremstyle{definition}                       
\theoremstyle{definition}                    
\theoremstyle{definition}                  
\theoremstyle{definition} \newtheorem{construction}[thm]{Construction}         
\theoremstyle{definition}              
\theoremstyle{definition}                  
\theoremstyle{definition} \newtheorem{prop-df}[thm]{Proposition-Definition}    

\theoremstyle{definition}
\newtheorem*{construction*}{Construction}                                      
\newtheorem*{conjecture*}{Conjecture}                                          
\newtheorem*{hypothesis*}{Hypothesis}                                          
\newtheorem*{convention*}{Convention}                                          
\newtheorem*{notation*}{Notation}                                              
\newtheorem*{prop*}{Proposition}                                               
\newtheorem*{summary*}{Summary}                                                
\newtheorem*{qt*}{Question}                                                    
\newtheorem*{rmk*}{Remark}                                                     
\newtheorem*{fact*}{Fact}                                                      
\newtheorem*{lizi*}{Example}                                                   
\newtheorem*{df*}{Definition}                                                  
\theoremstyle{plain}
\newtheorem*{thm*}{Theorem}                                                    

\crefname{thm}{Theorem}{Theorems}                                              %
\crefname{thme}{Th\'eo\`eme}{Th\'eo\`emes}
\crefname{lem}{Lemma}{Lemmas}
\crefname{lemme}{Lemme}{Lemmes}
\crefname{eg}{Example}{Examples}
\crefname{ege}{Exemple}{Exemples}
\crefname{rmk}{Remark}{Remarks}
\crefname{rmke}{Remarque}{Remarques}
\crefname{cor}{Corollary}{Corollaries}
\crefname{core}{Corollaire}{Corollaires}
\crefname{df}{Definition}{Definitions}
\crefname{dfe}{D\'efinition}{D\'efinitions}
\crefname{question}{Question}{Questions}
\crefname{prop}{Proposition}{Propositions}
\crefname{conjecture}{Conjecture}{Conjectures}
\crefname{construction}{Construction}{Constructions}

\newcommand{\bconst}{\begin{construction}}                                     
\newcommand{\econst}{\end{construction}}                                       
\newcommand{\benum}{\begin{enumerate}[label={{\upshape(\alph*)}}]}             
\newcommand{\benuma}{\begin{enumerate}[label={{\upshape(\arabic*)}}]}          
\newcommand{\benumr}{\begin{enumerate}[label={{\upshape(\roman*)}}]}           
\newcommand{\eenum}{\end{enumerate}}
\newcommand{\bconj}{\begin{conjecture}}
\newcommand{\econj}{\end{conjecture}}
\newcommand{\bconjnn}{\begin{conjecture*}}
\newcommand{\econjnn}{\end{conjecture*}}
\newcommand{\begs}{\begin{eg}\hfill\benuma}                                    
\newcommand{\eegs}{\eenum\end{eg}}                                             
\newcommand{\brmks}{\begin{rmk}\hfill\benuma}                                  
\newcommand{\ermks}{\eenum\end{rmk}}                                           
\newcommand{\bdfs}{\begin{df}\hfill\benuma}                                    
\newcommand{\edfs}{\eenum\end{df}}                                             
\newcommand{\bitem}{\begin{itemize}}                                           
\newcommand{\eitem}{\end{itemize}}                                             
\newcommand{\be}{\begin{equation}}                                             
\newcommand{\ee}{\end{equation}}                                               
\newcommand{\benn}{\begin{equation*}}                                          
\newcommand{\eenn}{\end{equation*}}                                            
\newcommand{\bqt}{\begin{qt*}\rm}                                              
\newcommand{\eqt}{\end{qt*}}                                                   
\newcommand{\bqtr}{\begin{qt*}\rm\coLR}                                        
\newcommand{\eqtr}{\end{qt*}}                                                  
\newcommand{\beac}{\begin{equation}\begin{array}{c}}                           
\newcommand{\eeac}{\end{array}\end{equation}}                                  
\newcommand{\beqn}{\begin{eqnarray*}}
\newcommand{\eeqn}{\end{eqnarray*}}
\newcommand{\bdf}{\begin{df}}
\newcommand{\bdfhf}{\begin{df}\hfill}
\newcommand{\edf}{\end{df}}
\newcommand{\brmk}{\begin{rmk}}
\newcommand{\brmkhf}{\begin{rmk}\hfill}
\newcommand{\ermk}{\end{rmk}}

\newcommand{\mrm}{\mathrm}

   \newcommand{\BR}{\mathbf{R}}

  \newcommand{\CF}{\mathcal{F}}

\newcommand{\CO}{\mathcal{O}}  
\newcommand{\CQ}{\mathcal{Q}}

\newcommand{\CY}{\mathcal{Y}}  

 \newcommand{\FX}{\mathfrak{X}}

\newcommand{\BBG}{\mathbb{G}}

\newcommand{\rma}{\mathrm{a}}
\newcommand{\rmb}{\mathrm{b}}

\newcommand{\olf}{\overline{f}}

\newcommand{\olK}{\overline{K}}
\newcommand{\olL}{\overline{L}}

\newcommand{\olU}{\overline{U}}

\newcommand{\olY}{\overline{Y}}

\newcommand{\olZ}{\overline{Z}}

\newcommand{\A}{\mathbb{A}}                                                    
\newcommand{\Z}{\mathbb{Z}}                                                    
\newcommand{\QZ}{\mathbb{Q}/\mathbb{Z}}                                        
\newcommand{\al}{\alpha}                                                       
\newcommand{\og}{\omega}                                                       

\newcommand{\Ra}{\Rightarrow}                                                  

\newcommand{\coLR}{\textcolor[rgb]{1.00,0,0}}                                  

\newcommand{\ol}{\overline}                                                    
\newcommand{\wt}{\widetilde}                                                   

\newcommand{\ce}{\colonequals}                                                 
\newcommand{\es}{\varnothing}                                                  
\newcommand{\uu}{^{\times}}                                                    
\newcommand{\uun}{^{(1)}}                                                      
\newcommand{\lip}{\langle}                                                     
\newcommand{\rip}{\rangle}                                                     

\newcommand{\drl}{\varinjlim}                                                  

\newcommand{\tstprodlim}{\tst\prod\limits}                                     

\DeclareMathOperator{\pic}{Pic}                                                
\DeclareMathOperator{\br}{Br}                                                  

\DeclareMathOperator{\Image}{Im}                                               
\renewcommand{\Im}{\Image}                                                     
\DeclareMathOperator{\Ker}{Ker}                                                
\renewcommand{\ker}{\Ker}                                                      

\DeclareMathOperator{\e}{Spec}                                                 
\DeclareMathOperator{\codim}{codim}                                            

\newcommand{\nr}{{\mathrm{nr}}}                                                

\DeclareMathOperator{\ots}{\otimes}                                            
\DeclareMathOperator{\otsZ}{\otimes_{\Z}}                                      

\newcommand{\gm}{\BBG_m}                                                       
\newcommand{\hnr}{H_{\nr}}                                                     

\newcommand{\cmdm}{commutative diagram~}

\newcommand{\distri}{distinguished triangle~}

\newcommand{\inpart}{In particular,~}

\newcommand{\ttes}{exact sequence~}

\newcommand{\ses}{short exact sequence~}

\newcommand{\wrt}{with respect to~}

\newcommand{\HSerre}{Hochschild--Serre~}

\newcommand{\TateSha}{Tate--Shafarevich~}                                      

\newcommand{\qand}{\quad\ \textup{and}\quad\ }                                 

\newcommand{\itm}{\item}
\newcommand{\tst}{\textstyle}

\newcommand{\pf}{\proof}
\newcommand{\pff}{\proof\hfill}
\newcommand{\vem}{\vspace{1em}}

\begin{document}
\title{\textbf{A Comparison of Cohomological Obstructions to the Hasse Principle and to Weak Approximation}}

\author{Yisheng TIAN}

\maketitle

\begin{abstract}
We show that certain \TateSha groups are unramified
which enables us to give an obstruction to the Hasse principle for torsors under tori over $p$-adic function fields.
\end{abstract}


\section{Introduction}
Let $K$ be the function field of a smooth projective geometrically integral curve $X$ over a $p$-adic field $k$.
In recent years,
Harari--Scheiderer--Szamuely \cite{HSSz15} obtained an explicit description of a cohomological obstruction
to weak approximation for $K$-tori
and also related this obstruction to an unramified cohomology group
(called the reciprocity obstruction).
In the meantime,
Harari--Szamuely \cite{HSz16} constructed an obstruction to the Hasse principle for $K$-torsors under tori.
By construction of \cite{HSz16}*{pp.~16-17 and Theorem 4.1},
the obstruction to the Hasse principle is essentially coming from the global duality for $K$-tori,
and they showed that this obstruction given by a group determined by "local triviality" is the only one.
However, the local triviality condition increases the difficulty of computing the obstruction.
This encourages us to find an obstruction which is easier to compute.

On the other hand,
the reciprocity obstruction is widely used in the investigation of the Hasse principle
and
intuitively the obstruction introduced in \cite{HSz16} should be compatible
with the reciprocity obstruction.
These known results also motivate us to build a connection between the obstruction in \cite{HSz16}
and the reciprocity obstruction.
Moreover, another benefit of this attempt is that the unramified cohomology group
is indeed easier to handle (for example, it vanishes in several known cases).
At this stage,
we would like to find out the relation between
the obstruction in \cite{HSz16}
and
the unramified cohomology group used in reciprocity obstruction \cite{HSSz15}*{Theorem 4.1}).

Actually the question that whether the canonical image of a certain map is unramified
was first raised by Colliot-Th\'el\`ene
(see also \cite{CT15}*{Remarque 4.3(b)}).
Later \cite{Tia21WA}*{Appendix} obtained a partial answer when the torsor is trivial.
Therefore it is an interesting question to describe the cohomological obstruction to the Hasse principle
using unramified cohomology groups for general torsors under tori.

Let us make the above statements more precise.
Let $X\uun$ be the set of closed points on $X$.
The local ring $\CO_{X,v}$ at $v\in X\uun$ is a discrete valuation ring,
so we may denote by $K_v$ the completion of $K$ with respect to $v\in X\uun$.
Let $Y$ be a $K$-torsor under a $K$-torus $T$.
We put $Y_v\ce Y\times_KK_v$.
Let $\CY$ be a smooth integral separated $X_0$-scheme such that $\CY\times_{X_0}\e K\simeq Y$ for some sufficiently small non-empty open subset $X_0\subset X$.
We define the adelic points on $Y$
(which does not depend on the choice of the model $\CY$) by
\[
Y(\A_K)\ce\drl_{U\subset X_0}\big(\tstprodlim_{v\notin U}Y(K_v)\times \tstprodlim_{v\in U}\CY(\CO_v)\big).
\]
In \cite{HSz16}*{page 15-16}, Harari and Szamuely constructed a map
$\rho_Y:H^3_{\mrm{lc}}(Y,\QZ(2))\to \QZ$
where
\[
H^3_{\mrm{lc}}(Y,\QZ(2))
\ce
\ker\big(\textstyle \frac{H^3(Y,\QZ(2))}{\Im H^3(K,\QZ(2))}\to \prod\limits_{v\in X\uun}\frac{H^3(Y_v,\QZ(2))}{\Im H^3(K_v,\QZ(2))}\big)
\]
gives the only obstruction to the Hasse principle as follows.

\begin{thm*}
[Harari--Szamuely, \cite{HSz16}*{Theorem 5.1}]
\label{result: obs to HP by Harari--Szamuely}
Let $Y$ be a $K$-torsor under a torus $T$ such that $Y(\A_K)\ne\es$.
If $\rho_Y$ is identically zero, then $Y(K)\ne\es$.
\end{thm*}

Actually, during the proof of the theorem,
Harari and Szamuely constructed a finer obstruction than that given by $H^3_{\mrm{lc}}(Y,\QZ(2))$.
Indeed,
the \HSerre spectral sequence
$
E^{p,q}_2\ce H^p(K,H^q(\ol{Y},\QZ(2)))\Ra H^{p+q}\ce H^{p+q}(Y,\QZ(2))
$
yields a map
$
H^2(K,H^1(\ol{Y},\QZ(2)))\to H^3(Y,\QZ(2))/\Im H^3(K,\QZ(2))
$
(where we have used implicitly the fact that the cohomological dimension of $K$ is $3$).
Moreover, \cite{HSz16}*{Lemma 5.2} allows one to identify
$H^2(K,H^1(\ol{Y},\QZ(2)))$
with
$H^2(K,T')$
where $T'$ is the dual torus of $T$
(i.e. $T'$ is the torus such that its module of characters is the module of cocharacters of $T$).
Let $\Sha^2(T')$ be the group of everywhere locally trivial elements of $H^2(K,T')$.
Restricting to the subgroup $\Sha^2(T')$ of $H^2(K,T')$ yields a map
\[
\tau:\Sha^2(T')\to H^3_{\mrm{lc}}(Y,\QZ(2)).
\]
Let $Y$ be a $T$-torsor such that $Y(\A_K)\ne\es$.
Since $Y(\A_K)\ne\es$ is equivalent to $Y(K_v)\ne\es$ for all $v\in X\uun$,
$Y(\A_K)\ne\es$ implies that the class $[Y]\in H^1(K,T)$ actually lies in $\Sha^1(T)$.
Now we arrive at:

\begin{prop*}
[Harari--Szamuely, \cite{HSz16}*{Proposition 5.3}]
Let $Y$ be a $T$-torsor such that $Y(\A_K)\ne\es$.
Then
\[
\rho_Y\circ\tau(\al)=\lip [Y],\al \rip
\]
holds up to sign for all $\al\in\Sha^2(T')$,
where $\lip-,-\rip$ is the global duality pairing $\Sha^1(T)\times \Sha^2(T')\to \QZ$ $($see \cite{HSz16}*{Theorem 4.1}$)$.
\end{prop*}

The previous theorem follows immediately from the precedent proposition
together with the fact that the global duality pairing $\Sha^1(T)\times \Sha^2(T')\to \QZ$ is perfect.
In this way, we obtain a cohomological obstruction to the Hasse principle given by the image of $\Sha^2(T')$ in $H^3_{\mrm{lc}}(Y,\QZ(2))$.
Thus the image of $\Sha^2(T')$ is the crucial part of the obstruction to the Hasse principle.

As for weak approximation,
Harari, Scheiderer and Szamuely announced that the defect to weak approximation for tori
can be described by $\Sha^2_{\og}(T')$,
where $\Sha^2_{\og}(T')$ denotes the subgroup of $H^2(K,T')$
consisting of elements vanishing in $H^2(K_v,T')$ for all but finitely many $v\in X\uun$.
To this end, they constructed a pairing (see \cite{HSSz15}*{\S4, pp.~18})
\[
(-,-)_{\mrm{WA}}:\tstprodlim_{v\in X^{(1)}}Y(K_v)\times \hnr^3(K(Y),\QZ(2))\to \QZ.
\]
Here $\hnr^3(K(Y),\QZ(2))$ denotes the unramified subgroup of $H^3(K(Y),\QZ(2))$
which contains the canonical image of $H^3(K,\QZ(2))$ in $H^3(K(Y),\QZ(2))$.
See \cite{CT95SBB}*{\S2 and \S4} for general properties of unramified elements and unramified cohomology, respectively.

\begin{thm*}
[Harari--Scheiderer--Szamuely, \cite{HSSz15}*{Theorem 4.2}]
There is a homomorphism
\[
u:\Sha^2_{\og}(T')\to \hnr^3(K(T),\QZ(2))
\]
such that each family $(t_v)\in T(K_v)$ annihilated by $(-,\Im u)_{\mrm{WA}}$
lies in the closure $\ol{T(K)}$ of $T(K)$ in $\prod T(K_v)$ \wrt the product topology.
\end{thm*}

With this point of view, we are interested in the image of
$H^3_{\mrm{lc}}(Y,\QZ(2))\to H^3(K(Y),\QZ(2))/H^3(K,\QZ(2))$.
We prove that this map has unramified image\footnote{Throughout, the unramified part of $H^3(K(Y),\QZ(2))/H^3(K,\QZ(2))$ will mean the quotient $\hnr^3(K(Y),\QZ(2))/H^3(K,\QZ(2))$.} and hence the obstruction to the Hasse principle will be easier to handle.
The first result is the following:

\begin{thm*}
[Tian, \cite{Tia21WA}*{Appendix}]
The image of $\Sha^2_{\og}(T')\to H^3(K(T),\QZ(2))/H^3(K,\QZ(2))$ is unramified,
i.e. its image lies in the subquotient $\hnr^r(K(T),\QZ(2))/H^3(K,\QZ(2))$.
\end{thm*}

In particular, the image of $\Sha^2(T')\to H^3_{\mrm{lc}}(T,\QZ(2))\to H^3(K(T),\QZ(2))/H^3(K,\QZ(2))$ is unramified.
This positive answer suggests us to generalize the result to $K$-torsors under $T$.
Now we arrive at the main result:

\begin{thm*}
[\cref{result: Sha-2 of tori non-ram general}]
The image of
$
\Sha^2_{\og}(T')\to H^3(Y,\QZ(2))/\Im H^3(K,\QZ(2))
$
is unramified for any $K$-torsor $Y$ under $T$.
\end{thm*}

The idea of the proof is simple:
we show that the image of $\Sha^2_{\og}(T')$ in $H^3(Y,\QZ(2))/\Im H^3(K,\QZ(2))$
comes from $H^3(Y^c,\QZ(2))/\Im H^3(K,\QZ(2))$
where $Y^c$ denotes a smooth compactification of $Y$
(see \cite{Kollar09}*{Theorem 3.21 and Corollary 3.22} for the existence of $Y^c$).
Thus the image of $\Sha^2_{\og}(T')$ in $H^3(K(Y),\QZ(2))/H^3(K,\QZ(2))$ is unramified
by \cite{CT95SBB}*{Proposition 2.1.8} and the properness of $Y^c$.
To this end, we use purity sequences to relate the cohomology of $Y$ with that of $Y^c$
via a commutative diagram (\cref{lemma: 2 by 2 by 3 diagram}).

\vem
\textbf{Acknowledgements}.
I thank my PhD supervisor David Harari for many useful discussions and helpful comments.
I thank Yang Cao for his help on the proof of the main theorem.
I thank the referees for valuable comments.
Finally, I thank Universit\'e Paris-Saclay and Southern University of Science and Technology for excellent conditions for research.

\section{Main results}
Let $T^c$ be a $T$-equivariant smooth projective compactification of $T$ over $K$
(see \cite{CHS05}*{Corollaire 1}).
Let $T\subset V_i\subset T^c$ be the open subset of $T^c$ consisting of $T$-orbits such that $\codim(T^c\setminus V_i,T^c)\ge i$.
Let $Y^c=Y\times^TT^c$ and let $U_i=Y\times^TV_i\subset Y^c$.
Fix an algebraic closure $\olK$ of $K$
and let $\ol{\FX}$ be the base change of a $K$-scheme $\FX$ to $\olK$.
So $U_0=Y$ by construction and $\ol{Y^c}$ is cellular by \cite{Cao18nr}*{Proposition 2.5(3)}.
For $i\ge 1$, $Z_i:=U_i\setminus U_{i-1}$ is smooth of codimension $i$ in $U_i$.

We begin with the computation of some cohomology groups via purity
and then
deduce a \cmdm which tells us $\Sha^2_{\og}(T')$ is unramified.
Throughout we shall simply write $\CQ(i)\ce \QZ(i)$ for $i\in\Z$.

\begin{lem}\label{lemma: inductive isomorphism by purity}
Suppose either $(\rma)$ $0\le r\le 4$ and $i\ge 2$,
or $(\rmb)$ $0\le r\le 2$ and $i\ge 1$.
There are isomorphisms
\[
H^r(U_i,\CQ(2))\simeq H^r(Y^c,\CQ(2))
\qand
H^r(\ol{U}_i,\CQ(2))\simeq H^r(\ol{Y^c},\CQ(2)).
\]
\end{lem}
\begin{proof}
The purity of $U_{i-1}\subset U_{i}\supset Z_{i}$ for $i\ge 1$ yields exact sequences (and similar sequences over $\ol{K}$)
\be\label{sequence: l-e-s for purity pair of codim i}
H^{r-2i}(Z_{i},\CQ(2-i))\to H^r(U_{i},\CQ(2))\to H^r(U_{i-1},\CQ(2))\to H^{r+1-2i}(Z_{i},\CQ(2-i))
\ee
by \cite{Fu11}*{Corollary 8.5.6}.
Thus for $r\le 4$ and $i\ge 3$, there are isomorphisms
\be\label{isomorphism: by purity}
H^r(U_{i},\CQ(2))\to H^r(U_{i-1},\CQ(2))
\ee
by the \ttes (\ref{sequence: l-e-s for purity pair of codim i}).
In particular, applying (\ref{isomorphism: by purity}) inductively yields isomorphisms
$H^r(U_2,\CQ(2))\simeq H^r(Y^c,\CQ(2))$ and $H^r(\ol{U}_2,\CQ(2))\simeq H^r(\ol{Y^c},\CQ(2))$ for $r\le 4$.
For case $(\rmb)$, since $\codim(Z_2,U_2)=2$,
the \ttes (\ref{sequence: l-e-s for purity pair of codim i}) for $0\le r\le 2$ and $U_1\subset U_2\supset Z_2$ implies
$H^r(U_2,\CQ(2))\simeq H^r(U_1,\CQ(2))$ and $H^r(\ol{U}_2,\CQ(2))\simeq H^r(\ol{U}_1,\CQ(2))$.
Therefore we conclude $H^r(U_1,\CQ(2))\simeq H^r(U_2,\CQ(2))\simeq H^r(Y^c,\CQ(2))$ and similarly over $\olK$.
\end{proof}

\begin{rmk}\label{remark: vanishing of odd degree cohomology}
Since $\ol{Y^c}$ is a projective cellular variety, $H^i(\ol{Y^c},\CQ(2))=0$ for $i=2n-1$ with $n\ge 1$ an integer
(see \cite{Fulton84}*{Example 19.1.11} or \cite{Cao18nr}*{Th\'eor\`eme 2.6}).
Moreover, thanks to $\br(\ol{Y^c})=0$ (since $\ol{Y^c}$ is smooth projective rational)
and the Kummer sequence
\[
0\to H^1(\ol{Y^c},\gm)/n\to H^2(\ol{Y^c},\mu_n)\to {_n}H^2(\ol{Y^c},\gm)\to 0,
\]
we obtain an isomorphism of Galois modules $H^2(\ol{Y^c},\CQ(1))\simeq\pic(\ol{Y^c})\otsZ\QZ$ after taking direct limit over all $n$.
Over the algebraically closed field $\olK$,
we can (non-canonically) identify $\CQ(1)$ with $\CQ(2)$,
so there is an isomorphism $H^2(\ol{Y^c},\CQ(2))\simeq\pic(\ol{Y^c})\otsZ\QZ$ of abelian groups.
\end{rmk}

\begin{lem}\label{lemma: purity I}\hfill
\benuma
\item\label{Pure1}
We have $H^1(\ol{U}_1,\CQ(2))=0$.
The map
$H^2(\ol{U}_1,\CQ(2))\to H^2(\ol{U}_0,\CQ(2))$ induced by $U_0\subset U_1$ is identically zero.
\inpart there is an \ttes of Galois modules:
\be\label{sequence: s-e-s of Galois modules I}
0\to H^1(\ol{U}_0,\CQ(2))\to H^0(\ol{Z}_1,\CQ(1))\to H^2(\ol{U}_1,\CQ(2))\to 0.
\ee

\item\label{Pure2}
We have
\[
\Im\big(H^3(Y^c,\CQ(2))\to H^3(U_1,\CQ(2))\big)
=\ker\big(H^3(U_1,\CQ(2))\to H^3(\ol{U}_1,\CQ(2))\big).
\]
Therefore
\be\label{identity: kernel v-s image}
\tst\Im\big(\frac{H^3(Y^c,\CQ(2))}{H^3(K,\CQ(2))}\to \frac{H^3(U_1,\CQ(2))}{H^3(K,\CQ(2))}\big)
=\ker\big(\frac{H^3(U_1,\CQ(2))}{H^3(K,\CQ(2))}\to H^3(\ol{U}_1,\CQ(2))\big).
\ee

\item\label{Pure3}
Consider the diagonal map $\Delta:H^2(K,H^0(\ol{Z}_1,\CQ(1)))\to \prod_vH^2(K_v,H^0(\ol{Z}_1,\CQ(1)))$
and write the image $(\al_v)\ce \Delta(\al)$ of $\al\in H^2(K,H^0(\ol{Z}_1,\CQ(1)))$ under $\Delta$ into a family of local elements.
Put
\[
\Sha^2_{\og}(H^0(\ol{Z}_1,\CQ(1)))
\ce\{\al\ |\ \text{$\al_v=0$ for all but finitely many $v\in X\uun$}\}.
\]
Then $\Sha^2_{\og}(H^0(\ol{Z}_1,\CQ(1)))=0$ and in particular $\Delta$ is injective.
\eenum
\end{lem}
\pff
\benuma
\item
By \cref{lemma: inductive isomorphism by purity} and \cref{remark: vanishing of odd degree cohomology},
we have $H^1(\ol{U}_1,\CQ(2))\simeq H^1(\ol{Y^c},\CQ(2))=0$.
Consider the following commutative diagram
\[
\xm@R10pt{
\pic(\ol{T^c})\ots\QZ\ar[r]\ar[d] & \pic(\ol{T})\ots\QZ\ar[d]\\
H^2(\ol{T^c},\CQ(2))\ar[r] & H^2(\ol{T},\CQ(2)).
}
\]
Here the vertical arrows are induced by the Kummer sequence $0\to \mu_n\to \gm\to \gm\to 0$
together with a fixed identification of $\CQ(1)$ and $\CQ(2)$ over $\olK$.
Note that $\br(\ol{T^c})=0$ and $\pic(\ol{T})=0$.
Since the left vertical arrow is an isomorphism by the vanishing of $\br(\ol{T^c})$,
we conclude that $H^2(\ol{Y^c},\CQ(2))\to H^2(\ol{U}_0,\CQ(2))$ is identically zero
because of the isomorphisms of varieties $\ol{Y^c}\simeq\ol{T^c}$ and $\ol{U}_0=\olY\simeq\ol{T}$.
Finally, the purity sequence (\ref{sequence: l-e-s for purity pair of codim i}) for $U_0\subset U_1\supset Z_1$ together with the above vanishing results imply the desired short exact sequence of Galois modules.

\item
The purity for $U_1\subset U_2\supset Z_2$ induces a commutative diagram with exact rows
\[
\xm@R10pt{
  H^3(U_2,\CQ(2))\ar[r]\ar[d] & H^3(U_1,\CQ(2))\ar[r]\ar[d] & H^0(Z_2,\CQ(0))\ar[d] \\
  H^3(\ol{U}_2,\CQ(2))\ar[r] & H^3(\ol{U}_1,\CQ(2))\ar[r] & H^0(\ol{Z}_2,\CQ(0)).
  }
\]
Note that the map $H^0(Z_2,\CQ(1))\to H^0(\ol{Z}_2,\CQ(1))$ is injective.
According to \cref{lemma: inductive isomorphism by purity} and \cref{remark: vanishing of odd degree cohomology}, we have
\[
H^3(\ol{U}_2,\CQ(2))\simeq H^3(\ol{Y^c},\CQ(2))=0.
\]
Thus a diagram chasing yields
\[
\Im\big(H^3(U_2,\CQ(2))\to H^3(U_1,\CQ(2))\big)=\ker\big(H^3(U_1,\CQ(2))\to H^3(\ol{U}_1,\CQ(2))\big).
\]
Recall that $H^3(U_2,\CQ(2))\simeq H^3(Y^c,\CQ(2))$ by \cref{lemma: inductive isomorphism by purity}, we are done.

\item
Note first that $H^0(\ol{Z}_1,\CQ(1))$ is isomorphic to a direct sum of copies of $\CQ(1)$ as Galois modules
and hence $\Sha^2_{\og}(H^0(\olZ_1,\CQ(1)))$ is a direct sum of copies of $\Sha_{\og}^2(\CQ(1))$.
For a field $L$ of characteristic zero,
we have a \ses $0\to \CQ(1)\to \olL\uu\to Q\to 0$ of Galois modules,
where the quotient $Q$ is uniquely divisible.
Taking cohomology yields an
isomorphism $H^2(L,\CQ(1))\simeq \br L$.
Finally,
applying $H^2(L,\CQ(1))\simeq \br L$ to $L=K, K_v$ and
taking $\Sha^2_{\og}(\gm)=0$ (\cite{HSSz15}*{Lemma 3.2(a)})
into account yield $\Sha^2_{\og}(\CQ(1))\simeq\Sha^2_{\og}(\gm)=0$.
\qed
\eenum

In the diagram below,
we denote by HU/HL, VF/VB/VL/VM/VR for the horizontal upper/lower, vertical front/back/left/middle/right face, respectively.
To avoid confusion, in VB we draw dashed vertical arrows.
So all the other faces are uniquely determined.

In the proof of \cref{result: Sha-2 of tori non-ram general}
we only need the commutativity of the left cube and the exactness of the upper row of HL in \cref{lemma: 2 by 2 by 3 diagram},
but we still construct and show the commutativity of the two cubes because all the involved arrows arise naturally. 

\begin{lem}\label{lemma: 2 by 2 by 3 diagram}
There is an exact \cmdm with surjective vertical arrows
\[
\xm@C=-30pt @R=10pt{
& H^3(K,\tau_{\le2}\BR \olf_{U_1*}\CQ(2)) \ar@{-->}[rr] \ar[ld] \ar@{-->}[dd]
&& H^3(K,\tau_{\le1}\BR \olf_{U_0*}\CQ(2)) \ar@{-->}[rr] \ar[ld] \ar@{-->}[dd]
&& H^2(K,\tau_{\le 0}\BR \olf_{Z_1*}\CQ(1)) \ar[ld] \ar@{-->}[dd]
\\
H^3(U_1,\CQ(2)) \ar[rr] \ar[dd]
&& H^3(U_0,\CQ(2)) \ar[rr] \ar[dd]
&& H^2(Z_1,\CQ(1)) \ar@{=}[dd] &
\\
& H^1(K,H^2(\ol{U}_1,\CQ(2))) \ar[rr] \ar[ld]
&& H^2(K,H^1(\ol{U}_0,\CQ(2))) \ar[rr] \ar[ld]
&& H^2(K,H^0(\ol{Z}_1,\CQ(1))) \ar[ld]
\\
\frac{H^3(U_1,\CQ(2))}{\Im H^3(K,\CQ(2))} \ar[rr]
&& \frac{H^3(U_0,\CQ(2))}{\Im H^3(K,\CQ(2))} \ar[rr]
&& H^2(Z_1,\CQ(1)). &
}
\]
\end{lem}
\begin{proof}
Let $j:\olU_0\to \olU_1$ and $i:\olZ_1\to \olU_1$ be open and closed immersions, respectively.
We consider the exact sequence
$0\to i_*i^!\CQ(2)\to \CQ(2)\to j_*j^*\CQ(2)\to i_*R^1i^!\CQ(2)\to 0$
of \'etale sheaves over $U_1$ (for example, see the proof of \cite{Fu11}*{Corollary 8.5.6}).
The isomorphism $R^{q}j_*\CQ(2)\simeq i_*R^{q+1}i^!\CQ(2)$ for $q\ge 1$ yields a distinguished triangle
\be\label{triangle: purity triangle I}
i_*Ri^!\CQ(2)\to \CQ(2)\to Rj_*j^*\CQ(2)\to i_*Ri^!\CQ(2)[1].
\ee
According to \cite{Fu11}*{Theorem 8.5.2} (or see \emph{loc. cit.} pp.~467 bottom for a more explicit version),
we have $R i^!\CQ(2)\simeq i^*\CQ(1)[-2]$.
Let $\olf_{\FX}$ denotes the structural morphism over $\olK$ of a  $K$-scheme $\FX$.
Then we obtain $\olf_{U_1}\circ i=\olf_{Z_1}$
and hence $\BR \olf_{U_1}i_*Ri^!\CQ(2)=\BR \olf_{Z_1}i^*\CQ(1)[-2]=\BR \olf_{Z_1}\CQ(1)[-2]$.
Now applying the functor $\BR \olf_{U_1*}$ to (\ref{triangle: purity triangle I})
yields a distinguished triangle
\be\label{triangle: purity triangle II}
\BR \olf_{Z_1*}\CQ(1)[-2]\to \BR \olf_{U_1*}\CQ(2)\to \BR \olf_{U_0*}\CQ(2)\to \BR \olf_{Z_1*}\CQ(1)[-1].
\ee
\bitem
\itm
The lower row of VB is obtained by taking Galois cohomology of (\ref{sequence: s-e-s of Galois modules I}).

\itm
We construct the upper row of VB.
The left upper horizontal arrow is constructed as follows.
Let $\CF$ be an \'etale sheaf over $\ol{\FX}$ for some $K$-scheme $\FX$.
Then we have  $H^2(\BR \olf_{\FX *} \CF)=R^2\olf_{\FX*}\CF\simeq H^2(\ol{\FX},\CF)$.
Consider the \cmdm of short exact sequences in the derived category of \'etale sheaves
\beac\label{diagram: constructing morphisms between truncated complexes}
\xm@C-10pt@R10pt{
0\ar[r]
& \wt{\tau}_{\le 1}\BR \olf_{U_1*}\CQ(2)\ar[r]\ar[d]
& \tau_{\le 2}\BR \olf_{U_1*}\CQ(2)\ar[r]\ar[d]\ar@{-->}[ld]
& H^2(\ol{U}_1,\CQ(2))[-2]\ar[r]\ar[d]
& 0 \\
0\ar[r]
& \wt{\tau}_{\le 1}\BR \olf_{U_0*}\CQ(2)\ar[r]
& \tau_{\le 2}\BR \olf_{U_0*}\CQ(2)\ar[r]
& H^2(\ol{U}_0,\CQ(2))[-2]\ar[r]
& 0
}
\eeac
where the vertical arrows are induced by the middle arrow of (\ref{triangle: purity triangle II}),
and the rows come from respective short exact sequences in the category of complexes of \'etale sheaves.
Since the right vertical map is zero by \cref{lemma: purity I}\ref{Pure1},
the arrow
$\tau_{\le 2}\BR \olf_{U_1*}\CQ(2)\to \tau_{\le 2}\BR \olf_{U_0*}\CQ(2)$
factors uniquely through
$\wt{\tau}_{\le 1}\BR \olf_{U_0*}\CQ(2)$
by the universal property of kernels.
But the quasi-isomorphism ${\tau}_{\le 1}\BR \olf_{U_0*}\CQ(2)\to \wt{\tau}_{\le 1}\BR \olf_{U_0*}\CQ(2)$
becomes an isomorphism in the derived category,
so we obtain a map
$\tau_{\le 2}\BR \olf_{U_1*}\CQ(2)\to {\tau}_{\le 1}\BR \olf_{U_0*}\CQ(2)$.

The right upper horizontal arrow is obtained by $\BR \olf_{U_0*}\CQ(2)\to \BR \olf_{Z_1*}\CQ(1)[-1]$ from (\ref{triangle: purity triangle II}).
Indeed, we have a map
\[
\tau_{\le 1}\BR \olf_{U_0*}\CQ(2)\to \tau_{\le 1}(\BR \olf_{Z_1*}\CQ(1)[-1])=(\tau_{\le 0}\BR \olf_{Z_1*}\CQ(1))[-1].
\]

\itm
The vertical arrows of VB are induced by the \distri
(see \cite{KSch06cat-and-sheaves}*{pp.~303, (12.3.2) and (12.3.3)})
\[
{\tau}_{\le j-1}\BR\ol{\pi}_*\CQ(2)
\to \tau_{\le j}\BR \ol{\pi}_*\CQ(2)
\to H^j(\ol{\FX},\CQ(2))[-j]
\to {\tau}_{\le j-1}\BR\ol{\pi}_*\CQ(2)[1]
\]
with $\ol{\pi}:\ol{\FX}\to \olK$ the structural morphism.
The map
\[
H^3(K,\tau_{\le2}\BR \olf_{U_1*}\CQ(2))\to H^1(K,H^2(\ol{U}_1,\CQ(2)))\]
is \emph{surjective}
since $H^4(K,\QZ(2))=0$.
Similarly, the other vertical arrows are also surjective.

\itm
Taking Galois cohomology of the triangle (\ref{triangle: purity triangle II}) yields an \ttes
\[
H^3(K,\BR f_{U_1*}\CQ(2))\to H^3(K,\BR f_{U_0*}\CQ(2))\to H^2(K,\BR f_{Z_1*}\CQ(1)),
\]
i.e. $H^3(U_1,\CQ(2))\to H^3(U_0,\CQ(2))\to H^2(Z_1,\CQ(1))$ is exact.
The horizontal rows of VF are constructed.

\itm
All the vertical arrows of VF are canonical projections.

\itm
Vertical arrows of HU
(those arrows of the form $H^i(K,\tau_{\le j}\BR\olf_{\FX*}\CF)\to H^i(\FX,\CF)$)
are induced by the canonical map $\tau_{\le i}A^*\to A^*$ of complexes,
and that of HL are obtained by the \HSerre spectral sequence
$H^p(K,H^q(\ol{\FX},\CQ(i)))\Ra H^{p+q}(\FX,\CQ(i))$,
where $\FX$ is a variety defined over $K$.
\eitem
Now we check the commutativity of the diagram.
\bitem

\itm
VL, VM and VR commute by construction of the \HSerre spectral sequence (see \cite{HSz16}*{pp.~17}).

\itm
VF commutes by construction.
For VB, we have a diagram in the derived category of \'etale sheaves
\[
\xm@C-10pt@R10pt{
\tau_{\le 2}\BR \olf_{U_1*}\CQ(2)\ar[r]\ar[d] & 
{\tau}_{\le 1}\BR \olf_{U_0*}\CQ(2)\ar[r]\ar[d] & 
\tau_{\le 0}\BR \olf_{Z_1*}\CQ(1)[-1]\ar[d] \\
H^2(\ol{U}_1,\CQ(2))[-2]\ar[r] & 
H^1(\ol{U}_0,\CQ(2))[-1]\ar[r] &  
H^0(\ol{Z}_1,\CQ(1))[-1]
}
\]
where the vertical arrows come from truncations.
This diagram commutes because the rows are both obtained from the triangle (\ref{triangle: purity triangle II}).
Hence applying the function $H^3(K,-)$ yields the commutativity of VB.

\itm
HU commutes by construction of truncation and by functoriality of Galois cohomology.
\eitem
Thus the horizontal lower face HL commutes by diagram chasing.
\end{proof}

\begin{cor}\label{result: Sha-2 of tori non-ram general}
The image of $\Sha^2_{\og}(T')$ in ${H^3(Y,\CQ(2))}/{\Im H^3(K,\CQ(2))}$ is unramified.
\end{cor}
\pf
By \cref{lemma: 2 by 2 by 3 diagram} and functoriality, the following diagram commutes:
\[
\xymatrix@R10pt{
& \Sha^2_{\og}(T')\ar[r]\ar[d]_{\iota_T} & \Sha^2_{\og}(H^0(\ol{Z}_1,\CQ(1)))\ar[d]_{\iota_Z}\\
H^1(K,H^2(\ol{U}_1,\CQ(2)))\ar[r]^-{\Phi}\ar[d]_{\mrm{HS}_1} &
H^2(K,H^1(\ol{U}_0,\CQ(2)))\ar[r]\ar[d]_{\mrm{HS}_0} & H^2(K,H^0(\ol{Z}_1,\CQ(1)))\ar[d]\\
\frac{H^3(U_1,\CQ(2))}{\Im H^3(K,\CQ(2))}\ar[r]_{\Psi} &
\frac{H^3(Y,\CQ(2))}{\Im H^3(K,\CQ(2))}\ar[r] & H^2(Z_1,\CQ(1)).
}
\]
Here $\iota_T$ and $\iota_Z$ are respective inclusions.
According to \cref{lemma: purity I},
the second row is exact and $\Sha^2_{\og}(H^0(\ol{Z}_1,\CQ(1)))=0$.
Subsequently a diagram chasing shows that $\Im(\iota_T)\subset\Im \Phi$.
By commutativity of the left lower square,
$\Im (\mrm{HS}_0\circ\iota_T)\subset \Im (\mrm{HS}_0\circ\Phi)=\Im (\Psi\circ \mrm{HS}_1)$.
The composite map
\[
H^1(K,H^2(\ol{U}_1,\CQ(2)))\to \tst\frac{H^3(U_1,\CQ(2))}{\Im H^3(K,\CQ(2))}\to H^3(\ol{U}_1,\CQ(2))
\]
factors through $H^1(\olK,H^2(\ol{U}_1,\CQ(2)))=0$ by the functoriality of the Hochschild--Serre spectral sequence.
Thus we obtain from \cref{lemma: purity I}\ref{Pure2} that
\[
\Im \mrm{HS}_1
\subset
\ker\big(\tst\frac{H^3(U_1,\CQ(2))}{\Im H^3(K,\CQ(2))}\to H^3(\ol{U}_1,\CQ(2))\big)
=
\Im\big(\tst\frac{H^3(Y^c,\CQ(2))}{H^3(K,\CQ(2))}\to \frac{H^3(U_1,\CQ(2))}{H^3(K,\CQ(2))}\big).
\]

Thus the image of $\Sha^2_{\og}(T')$ in $H^3(Y,\CQ(2))/\Im H^3(K,\CQ(2))$ comes from $H^3(Y^c,\CQ(2))$,
i.e. $\Sha^2_{\og}(T')$ is unramified.
\inpart the image of $\Sha^2(T')$ in $H^3_{\mrm{lc}}(Y,\CQ(2))$ is unramified.
\qed

\vem
If we restrict ourselves to the subgroup $\Sha^2(T')$ of $\Sha^2_{\og}(T')$,
then the image lies in the subgroup $H^3_{\mrm{lc}}(Y,\QZ(2))$ of $H^3(Y,\QZ(2))/\Im H^3(K,\QZ(2))$.
Now the image of $\Sha^2(T')$ in $H^3(K(Y),\QZ(2))/H^3(K,\QZ(2))$ is unramified by \cref{result: Sha-2 of tori non-ram general}.

Southern University of Science and Technology, No.1088 Xueyuan Blvd, Shenzhen, Canton, China 518055\\
\indent E-mail address: tianys@sustech.edu.cn

\end{document}